\documentclass[11pt]{article}
\usepackage[papersize={210truemm,297truemm},margin=3.0truecm,nohead]{geometry}
\usepackage{wrapfig}
\usepackage{mathrsfs}
\usepackage{amsthm}
\usepackage{amsmath}
\usepackage{amssymb}
\usepackage{amsfonts}
\usepackage{latexsym}
\usepackage{comment}
\usepackage{ascmac}
\usepackage{caption}
\usepackage{mathtools}
\usepackage[pdftex]{graphicx}  
\theoremstyle{plain}
\newtheorem{thm}{Theorem}[section]
\newtheorem*{syuA}{\textbf{Main Theorem A}}
\newtheorem*{syuB}{\textbf{Main Theorem B}}

\newtheorem{lem}[thm]{Lemma}
\newtheorem{cor}[thm]{Corollary}

\newtheorem{pro}[thm]{Proposition}
\theoremstyle{definition}
\newtheorem{df}[thm]{Definition}

\newtheorem{rem}[thm]{Remark}

\newtheorem{ex}[thm]{Example}

\newtheorem*{prf*}{Proof}
\newtheorem*{pf*}{}
\newtheorem*{conne*}{Argument for the connectedness}
\newtheorem*{lconne*}{Argument for the local connectedness}
\newtheorem*{lem*}{LemmaA}
\newtheorem*{lm*}{LemmaB}
\newtheorem*{stra*}{Strategy for the proof of main result A}

\def\g2{l\ge2}

\def\la{\lambda}
\def\mn{\mathcal{M}_n}

\def\on{\Omega_n}
\title{Mandelbrot set for fractal $n$-gons and zeros of power series
}
\author{Yuto Nakajima\\ Keio Institute of Pure and Applied Sciences (KiPAS),\\ Department of Mathematics,
Keio University, Yokohama,
223-8522, JAPAN\\
nakajimayuto@math.keio.ac.jp
}
\date{}

\begin{document}
\maketitle
\begin{abstract}
We give a framework to study the connectedness of the set of zeros of power series with coefficients in a finite subset $G\subset \mathbb{C}$. We prove that the set of zeros in the unit disk is connected and locally connected if some graph on the set $G$ of coefficients is connected. Furthermore, we apply this result to the study of the Mandelbrot set $\mathcal{M}_n$ for fractal $n$-gons. We prove that $\mathcal{M}_n$ is connected and locally connected for any $n$.
\end{abstract}

\section{Introduction}
\subsection{Background}
In 1985, Barnsley and Harrington \cite{BH} introduced 
the connectedness locus $\mathcal{M}_2$ for a pair of linear maps as an analog of the Mandelbrot set for quadratic polynomials, that is,  
\begin{align*}
\mathcal{M}_2=\{\la\in \mathbb{D}^{\times}\ :\ A_2(\la)\ \mbox{is connected} \},
\end{align*}
where $\mathbb{D}^{\times}:=\{\lambda\in \mathbb{C}\ :\ 0< |\lambda|< 1\}$ and the set $A_2(\la)$ is the attractor of the iterated function system $\{z\mapsto \la z+1, z\mapsto \la z-1\}$ on the set $\mathbb{C}$ of complex numbers. For the general theory of the iterated function system, see \cite{Fal}. 
Barnsley and Harrington \cite{BH} proved that there is a neighborhood $U$ of the set $\{0.5, -0.5\}$ such that $U\cap \mathcal{M}_2\subset\mathbb{R}$. Furthermore, they conjectured that there is a non-trivial hole in  $\mathcal{M}_2,$ which was confirmed by Bandt \cite{Ban1} in 2002. Bandt \cite{Ban1} also conjectured that the interior of $\mathcal{M}_2$ is dense away from $\mathcal{M}_2\cap \mathbb{R}$, that is, ${\rm cl}\bigr({\rm int}(\mathcal{M}_2)\bigl)\cup (\mathcal{M}_2\cap \mathbb{R})=\mathcal{M}_2$. Here, for a set $A\subset{\mathbb{C}}$, we denote by ${\rm cl}(A)$ and ${\rm int}(A)$ the closure of $A$ and the interior of $A$ with respect to the Euclidean topology on $\mathbb{C}$ respectively. Solomyak and Xu \cite{SX} made partial progress on Bandt's conjecture.

In 2008, Bandt and Hung \cite{BanHu} introduced self-similar sets parameterized by $\la\in \mathbb{D}^{\times}$ which are called ``{\em fractal $n$-gons}'' for $n\in \mathbb{N}$ with $n\ge 2$. We give the rigorous definition of ``{\em fractal $n$-gons}'' in the next sub-section (see Definition \ref{ngon}). They studied the connectedness locus for ``{\em fractal $n$-gons}'', that is, 
\begin{align*}
\mathcal{M}_n=\{\la\in \mathbb{D}^{\times}\ :\ A_n(\la)\ \mbox{is connected} \},
\end{align*}
where $A_n(\la)$ is the ``{\em fractal $n$-gon}'' corresponding to the parameter $\la$. Note that ``{\em fractal $2$-gons}'' are attractors of the iterated function systems $\{z\mapsto \la z+1, z\mapsto \la z-1\}$ and $\mathcal{M}_2$ is the connectedness locus for ``{\em fractal $2$-gons}''. Bandt and Hung \cite{BanHu} discovered many remarkable properties about $\mathcal{M}_n$, including the following result. For each $n\ge 3$ with $n\neq 4$, $\mn$ is regular-closed, that is, ${\rm cl}\bigr({\rm int}(\mathcal{M}_n)\bigl)=\mathcal{M}_n$. In 2016, Calegari, Koch and Walker \cite{CKW} introduced new methods for constructing interior points and positively answered Bandt's conjecture, that is, ${\rm cl}\bigr({\rm int}(\mathcal{M}_2)\bigl)\cup (\mathcal{M}_2\cap \mathbb{R})=\mathcal{M}_2$. Himeki and Ishii \cite{HI} proved $\mathcal{M}_4$ is regular-closed. Thus the problems about the regular-closedness of $\mn$ have been completely solved. 

We now consider the connectedness of $\mn$. Bousch \cite{Bou, Bou2} proved that $\mathcal{M}_2$ is connected and locally connected. This is interesting since for the case of quadratic polynomials, the local connectedness of the Mandelbrot set is still an open problem. Furthermore, Bousch \cite{Bou3} proved that $\mn$ is connected for any $n\ge 3$. In this paper, we obtain somewhat stronger results by proving the local connectedness of $\mn$ for any $n\ge 3$. In order to prove that, we consider a general framework to study the connectedness of the set of zeros of power series. The novelty of our framework is to
introduce a graph on the set of possible coefficients, ensuring that the corresponding set of zeros is connected and locally connected.
\subsection{Main results}
Bandt and Hung \cite{BanHu} introduced {\em fractal $n$-gons} and their  {\em Mandelbrot set} as follows.
\begin{df}[Fractal $n$-gons]
\label{ngon}
Let $\mathbb{D}^{\times}:=\{\lambda\in \mathbb{C}\ :\ 0< |\lambda|< 1\}$ be the parameter space. Fix a parameter $\lambda\in \mathbb{D}^{\times}$ and a natural number $n$ with $n\ge 2.$  We set ${\xi_n}=\exp(2\pi\sqrt{-1} /n)$. For each $i\in \{0,1,...,n-1\}$, we define $\phi^{n,\lambda}_i: \mathbb{C}\rightarrow \mathbb{C}$ by $\phi^{n,\lambda}_i(z)=\lambda z+{\xi_n}^i$. Then for the iterated function system $\{\phi^{n,\lambda}_0,..., \phi^{n,\lambda}_{n-1}\}$, there uniquely exists a non-empty compact subset $A_n(\lambda)$ such that $$\bigcup_{i=0}^{n-1}\phi^{n,\lambda}_i(A_n(\lambda))=A_n(\lambda)$$(see \cite{Fal, Hut}). We call $A_n(\lambda)$ a fractal $n$-gon corresponding to the parameter $\la$.
\end{df}
For each $n\in \mathbb{N}$ with $n\ge 2$, we define the Mandelbrot set $\mathcal{M}_n$ for fractal $n$-gons as 
\begin{align*}
\mathcal{M}_n=\{\la\in \mathbb{D}^{\times}\ :\ A_n(\la)\ \mbox{is connected} \}.
\end{align*}
Then the following theorem holds.
\begin{syuA}
For any $n\ge 2$, $\mn$ is connected and locally connected.
\end{syuA}
Main Theorem A overlaps with the work of Bousch \cite{Bou3, Bou2}. However, in case $n\ge 3$ the local connectedness of $\mn$ is a new result. 

We can identify $\mn$ with the set of zeros of some power series (see \cite[Remark 3]{BanHu}). Hence  in order to prove Main Theorem A we give the following setting, which provides a framework to study the connectedness of the set of zeros of power series.


Let $G$ be a non-empty finite subset of $\mathbb{C}$. We set $\triangle G:=\{a-b\in \mathbb{C}\ :\ a, b\in G\}.$ For $a\in \mathbb{C},$ we set $a G:=\{ab\in \mathbb{C}\ :\ b\in G\}.$ Moreover, we set 
$$R_G=\{a\in \mathbb{C}\ :\ aG\subset \triangle G\}.$$
We define a reflexive and symmetric relation $\mathcal{R}_G$ over $G$ as $a\mathcal{R}_Gb$ if and only if $a-b\in R_{G}$ for $a,b \in G.$ We now define an undirected graph on $G$ as follows.
\begin{df}
\label{graph}
Let $G$ be a non-empty finite subset of $\mathbb{C}.$ Let $(G, E_G)$ be the finite and undirected graph with the vertex set $G$ and the edge set $E_G$ of unordered pairs $(a, b)$ of elements of $G$ satisfying $a\mathcal{R}_Gb.$
\end{df} 
Below, we write $(G, R_G)$ for the graph $(G, E_G)$ for the emphasis on the set $R_G$. Moreover, we define the {\em connectedness} of the graph $(G, R_G)$ as follows.
\begin{df}
\label{graph1}
We say that the graph $(G, R_G)$ is {\em connected} if for any $a, b\in G$ with $a\neq b$ there exist $c_1,..., c_k\in G$ satisfying the following:
\begin{itemize}
\item[(i)] $c_1=a, c_k=b;$ 
\item[(ii)]the graph $(G, R_G)$ has an edge between $c_i$ and $c_{i+1},$ that is, $c_i-c_{i+1}\in R_G$ for any $i=1,..., k-1.$
\end{itemize}
\end{df} 
For each non-empty finite subset $G$ of $\mathbb{C}$, we define a set $P^{G}$ of functions on the open unit disk $\mathbb{D}=\{z\in \mathbb{C}\ :\ |z|< 1\}$ and the set $X^{G}$ of zeros of functions which belong to $P^{G}$ as follows.
\begin{df}
\label{poly}
\begin{align*}
&P^{G}=\left\{1+\sum_{i=1}^{\infty}a_iz^i\ :\ a_i\in G\ \text{for any}\ i=1, 2,... \right\},\\
&X^{G}=\{z\in \mathbb{D}\ :\ \mbox{there exists}\ {f\in P^{G}}\ \mbox{such that}\ f(z)=0\}.
\end{align*}
\end{df}

Then the following theorem holds, which is the second main result in this paper.
\begin{syuB}
Let $G\subset \mathbb{C}$ be a finite subset which contains $1$. Suppose that the graph $(G, R_G)$ is connected. If there exists a real number $L$ with $0< L <1$ such that $\{z\in \mathbb{C}\ :\ L< |z|<1\}\subset X^G,$ then we have $X^G$ is connected and locally connected.
\end{syuB}
\begin{rem}
In the case $G=\{-1, 1\}$ or $G=\{-1, 0, 1\}$, Bousch \cite{Bou, Bou2} proved that $X^{G}$ is connected and locally connected. In the case $G=\{0, 1\}$, Odlyzko and Poonen \cite{OP} proved that $X^{G}$ is path-connected.  
The method of Bousch \cite{Bou2} was inspired by Odlyzko and Poonen \cite{OP} and modified by Bandt \cite[Section 11]{Ban1}. Later, Sirvent and Thuswaldner \cite{ST} developed the method of Bousch \cite{Bou2} by using automata theory in the study of the connectedness locus of some iterated function systems.
For an important variation of $X^G,$ see \cite{SS}.
\end{rem}
The rest of this paper is devoted to proofs of Main Theorems A and B.

\subsection*{Acknowledgement.} 
The author would like to express his gratitude to anonymous referees for their valuable comments. He would like to express his gratitude to Hiroki Sumi and Takayuki Watanabe for their valuable discussions. 

\section{Proof of Main Theorem B}
In this section we give a proof of Main Theorem B. Below we fix a finite subset $G$ of $\mathbb{C}$ satisfying the assumptions of Main Theorem B. Set $G^0=\{\varepsilon\}$ and $G^{\ast}=\cup_{m\ge 0}G^m$, where $\varepsilon$ is the empty word. For any finite word $u=u_1\cdots u_m\in G^{\ast}\backslash{G^{0}}$ we define a subset $P_u^{G}\subset P^G$ as

\begin{align*}
&P_u^{G}=\left\{1+\sum_{i=1}^{\infty}a_iz^i\in P^G\ :\ a_i=u_i\ \text{for any}\ i=1,..., m \right\}.
\end{align*}
For $\varepsilon\in G^0,$ set $P_{\varepsilon}^{G}=P^G$.

We give a key lemma in this paper.
\begin{lem}
\label{keylem}
Let $u\in G^{\ast}.$ If the graph $(G, R_G)$ has an edge between $a$ and $b$, then there exist $p_{ua}\in P^{G}_{ua}$ and $p_{ub}\in P^{G}_{ub}$ such that every zero of $p_{ua}$  in $\mathbb{D}$ is also a zero of $p_{ub}.$ 
\end{lem}
\begin{proof}
Let $m\in \mathbb{N}$ and let $u=u_1\cdots u_m\in G^m.$ Take $a, b\in G$ such that the graph $(G, R_G)$ has an edge between $a$ and $b$.

Since $b-a\in R_G$, we have $(b-a)u_1\in \triangle G$ for $u_1\in G$, that is, for $u_1\in G$ there exists $c^1_1\in G$ such that $$(b-a)u_1+c^1_1\in G.$$ Similarly, for $u_i\in G$ there exists $c^1_i\in G$ such that $$(b-a)u_i+c^1_i\in G,$$ where $i\in\{1, 2,..., m+1\}$ and we set $u_{m+1}=a.$ Inductively, for $c^j_i\in G$ there exists $c^{j+1}_i\in G$ such that 
\begin{align*}
&(b-a)c^j_i+c^{j+1}_i\in G,
\end{align*}
where $i\in\{1, 2,..., m+1\}$ and $j\in \mathbb{N}$. Then we set 
\begin{align*}
p_{ua}:=(1, \underbrace{u_1,..., u_{m}, a}_{m+1}, \underbrace{c^1_1, ..., c^1_{m+1}}_{m+1}, \underbrace{c^2_1,..., c^2_{m+1}}_{m+1},\cdots, \underbrace{c^{j}_1,..., c^{j}_{m+1}}_{m+1}, \cdots)\in P_{ua}^{G}.
\end{align*}
Here, we denote by $(1, a_1, a_2,...)$ the power series $1+\sum_{i=1}^{\infty} a_i z^i$. 
Moreover, we set 
\begin{align*}
p_{ub}:=&\{1+(b-a)z^{m+1}\}p_{ua}\\
=&(\underbrace{1, u_1,..., u_{m}}_{m+1}, a, \underbrace{c^1_1, ..., c^1_{m+1}}_{m+1}, ..., \underbrace{c^{j+1}_1,..., c^{j+1}_{m+1}}_{m+1}, ...)+\\
&(\underbrace{0, 0,..., 0}_{m+1}, (b-a), \underbrace{(b-a)u_1,...,(b-a)a}_{m+1},...,\underbrace{(b-a)c^{j}_1,...,(b-a)c^{j}_{m+1}}_{m+1},... )\\
=&(1, u_1,..., u_{m}, b,..., \underbrace{(b-a)c^{j}_{1}+c^{j+1}_1,..., (b-a)c^{j}_{m+1}+c^{j+1}_{m+1}}_{m+1},...)\in P_{ub}^{G}.\\
\end{align*}
Hence the functions $p_{ua}$ and $p_{ub}$ satisfy the desired properties.
\end{proof}

 

Below, for $\gamma\in G^{\mathbb{N}},$ we denote by $f_{\gamma}$ the power series \[1+\sum_{i=1}^{\infty}\gamma_i z^i\] in $P^G.$

\begin{conne*}

\end{conne*}
We improve the methods in \cite{Bou}. For any $\gamma=\gamma_1\gamma_2\cdots, \delta=\delta_1\delta_2\cdots \in G^{\mathbb{N}},$ we set ${\rm Val}(f_{\gamma}, f_{\delta}):={\rm inf}\{i\in \mathbb{N}\ :\ \gamma_i\neq \delta_i\}.$ 
We set $\mathbb{N}_{\ge 2}:=\{n\in \mathbb{N}\ :\ n\ge 2\}.$

\begin{df}
\label{join}
Let $f, g\in P^G$ with $f\neq g$. 
Let $S=\{p_0, q_0, p_1, q_1, ..., p_{m}, q_{m}\}$ be a sequence of functions which belong to $P^G$. Let $N\in \mathbb{N}_{\ge 2}.$ We say that $S$ is a sequence of functions which joins $f$ to $g$ with respect to $N$ if $S$ satisfies the following:
\begin{enumerate}
\item[(1)]for each $i$, ${\rm Val}(p_i, q_i)\ge N;$
\item[(2)]for each $i$, every zero of $q_i$ in $\mathbb{D}$ is also a zero of $p_{i+1};$
\item[(3)]$p_0=f, q_m=g$.
\end{enumerate}
\end{df}
We identify $(1, a_1, a_2,...)$ with the power series $1+\sum_{i=1}^{\infty} a_i z^i$. Then the following holds.
\begin{lem}
\label{js}
Let $N\in \mathbb{N}_{\ge 2}.$ Then for any $f, g\in P^G$ with $f\neq g,$ there exists a sequence $p_0, q_0, p_1, q_1, ..., p_{m}, q_{m}$ of functions which joins $f$ to $g$ with respect to $N$.
\end{lem}
\begin{proof}
If ${\rm Val}(f, g)\ge N,$ we set $p_0=f$ and $q_0=g.$ Hence we assume ${\rm Val}(f, g)\in \{1,..., N-1\}.$ 

This is done by induction with respect to ${\rm Val}(f, g)\in \{1,..., N-1\}.$ We first prove that the statement holds in the case ${\rm Val}(f, g)=N-1,$ that is, $f, g$ have the following forms.
\begin{align*}
&f:=(1, a_1,..., a_{N-2}, a, *\cdots*),\\
&g:=(1, a_1,..., a_{N-2}, b, *\cdots*),
\end{align*}
where $a, b\in G$ with $a\neq b.$ It suffices to construct a sequence of functions which joins $f$ to $g$ in the case the graph $(G, R_G)$ has an edge between $a$ and $b$ since the graph $(G, R_G)$ is connected.
Then by Lemma \ref{keylem}, there exist $q_0\in P^G_{a_1\cdots a_{N-2}a}$ and $p_1\in P^G_{a_1\cdots a_{N-2}b}$ such that every zero of $q_0$ in $\mathbb{D}$ is also a zero of $p_{1}.$ Since ${\rm Val}(f, q_0)\ge N$ and ${\rm Val}(p_1, g)\ge N,$ we find the sequence $f, q_0, p_1, g$ which joins $f$ to $g$.

Fix $j\in \{1,..., N-2 \}.$ Suppose that the statement holds in the case ${\rm Val}(f, g)> j.$ We now prove that the statement holds in the case ${\rm Val}(f, g)=j.$ We set 
\begin{align*}
&f:=(1, a_1,..., a_{j-1}, a, *\cdots*),\\
&g:=(1, a_1,..., a_{j-1}, b, *\cdots*),
\end{align*}
where $a, b\in G$ with $a\neq b.$ We can assume the graph $(G, R_G)$ has an edge between $a$ and $b$. Then by Lemma \ref{keylem}, there exist $q_0\in P^G_{a_1\cdots a_{j-1}a}$ and $p_1\in P^G_{a_1\cdots a_{j-1}b}$ such that every zero of $q_0$ in $\mathbb{D}$ is also a zero of $p_{1}.$ Since ${\rm Val}(f, q_0)>j$ and ${\rm Val}(p_1, g)>j,$ by induction hypothesis, there exist sequences $S_1$ and $S_2$ of functions which join $f$ to $q_0$ and $p_1$ to $g$ respectively. Hence we find a sequence $S_1,S_2$ of functions which joins $f$ to $g$. Thus we have proved our lemma. 

\end{proof}
Let $\mathcal{O}(\mathbb{D})$ be the set of holomorphic functions on $\mathbb{D}.$ Set $$F:=\left\{(f, s)\in P^G\times {\rm cl}(B(0, L))\ :\ f(s)=0\right\}\subset \mathcal{O}(\mathbb{D})\times \mathbb{D},$$ where $L$ satisfies $\{z\in \mathbb{C}\ :\ L< |z|<1\}\subset X^G.$ Since $P^G$ is a compact subset of $\mathcal{O}(\mathbb{D})$ endowed with the compact open topology, $F$ is a compact subset of $\mathcal{O}(\mathbb{D})\times \mathbb{D}$. 
By using this fact and Rouch\'e's Theorem, we can give the following. 
\begin{lem}
\label{sr}
For any $\epsilon>0$ with $L+\epsilon<1	$, there exists $N_{\epsilon}\in \mathbb{N}_{\ge2}$ such that for all $(f, s)\in F$ and for all $g\in P^G$ with ${\rm Val}(f, g)\ge N_{\epsilon},$ there exists $s^{\prime}\in B(s, \epsilon)$ such that $g(s^{\prime})=0.$
\end{lem}
Then the following holds.
\begin{pro}
$X^G$ is connected.
\end{pro}
\begin{proof}
Suppose that ${X}^G$ is not connected, that is, $X^G=D\cup E$, where $D$ and $E$ are non-empty disjoint open and closed sets. We assume that the annulus $\{z\in \mathbb{C}\ :\ L< |z|<1\}$, which is a connected set in $X^G$, is contained in $D$. Set \[\epsilon=\min\left\{\inf\{|x-y|\ :\ x\in D, y\in E\}, \frac{1-L}{2}\right\}.\] Since $E$ is not empty, take $z\in E.$ Then there exists $\gamma\in G^{\mathbb{N}}$ such that $f_{\gamma}(z)=0.$ We set $g(z)=1+\sum_{i=1}^\infty z^i$ for any $z\in \mathbb{D}$. Then ${g}\in P^G$ since $G$ contains $1$.

Since $f_{\gamma}\neq g$, by Lemma \ref{js}, there exists a sequence of functions $p_0, q_0, p_1, q_1, ..., p_{m}, q_{m}$ which joins $f_{\gamma}$ to $g$ with respect to $N_{\epsilon},$ where $N_{\epsilon}$ is defined by $\epsilon$ in Lemma \ref{sr}. Then by Definition \ref{join} (1), we have ${\rm Val}(p_0, q_0)\ge N_{ \epsilon}.$ Hence Lemma \ref{sr} implies there exists $z^{\prime} \in B(z, \epsilon)$ such that $q_0(z^{\prime})=0.$ If $z^{\prime}\in \{z\in \mathbb{C}\ :\ L< |z|<1\}(\subset D),$ this contradicts the definition of $\epsilon,$ and hence we have ${X}^G$ is connected. Otherwise, by Definition \ref{join} (2) we have $p_1(z^{\prime})=0.$ If we repeat this procedure, there exist $s\in E,$ and $t\in D$ such that $t\in B(s, \epsilon)$ since $g(=q_m)$ does not have any roots in $\mathbb{D}.$ However, this contradicts the definition of $\epsilon.$ Hence we have proved our theorem.
\end{proof}

\begin{lconne*}

\end{lconne*}

We apply the method in \cite[p. 1142 Our modification]{Ban1} to our planar setting. Let $\eta=(1+L)/2.$ Since $\eta<1,$ by imitating the proof of \cite[Proposition 2.1]{OP}, there is a uniform constant $M$ such that for any $f\in P^G$ has at most $M$ roots counted with multiplicity in the disk $B(0, \eta).$ 
Fix $\epsilon>0$ with $\epsilon<\eta-L$. 
For any finite word $u=u_1\cdots u_m\in G^{m},$ we define a cylinder set $C_u$ as \[C_u=\left\{\omega\in G^{\mathbb{N}}\ :\ \omega_i=u_i\ \text{for any}\ i=1,..., m\right\}.\]

Define a set valued map $\Psi$ on $G^{\mathbb{N}}$ by \[\Psi(\gamma)=\{z\in \mathbb{D}\ :\ f_{\gamma}(z)=0\}.\] In the case $\{z\in B(0, \eta)\ :\ f_{\gamma}(z)=0\}\neq \emptyset$ for $\gamma\in G^{\mathbb{N}},$ let $r^1_{\gamma},..., r^l_{\gamma}$ be roots in $B(0, \eta)$ of $f_{\gamma},$ where $r^j_{\gamma}$ is a zero of multiplicity $k_j.$ Note that $\sum_{j=1}^l k_j\le M,$ where $M$ is the uniform bound. Let $U_j$ be a neighborhood of $r^j_{\gamma}$ such that the diameter of $U_j$ is less than $\epsilon$ and $r^j_{\gamma}$ is the unique root of $f_{\gamma}$ on $U_j$ and let $U_{\ast}:=\{z\in \mathbb{D}\ :\ \eta-\epsilon<|z|\}.$ By Rouch\'e's Theorem, there exists an initial word $u=u(\gamma)=\gamma_1\cdots \gamma_m\in G^m$ such that for any $\delta\in C_u,$\[\Psi(\delta)\subset \bigcup_{j=1}^l U_j\cup U_{\ast}.\]

Then define a 
continuous set-valued map $\psi_j: C_{u}\rightarrow \{F\subset U_j\ :\  \sharp F \le k_j\}$ by $\psi_j(\delta)=\{z\in U_j\ :\ f_{\delta}(z)=0\}$ for each $U_j$, and also define a set-valued map $\psi_{\ast}$ on $C_{u}$ by $\psi_{\ast}(\delta)=\{z\in U_{\ast}\ : f_{\delta}(z)=0\}.$ Hence $\Psi$ is decomposed into $\psi_j$ and $\psi_{\ast},$ that is, 
\begin{align}
\label{spl1}
\Psi(\delta)=\bigcup_{j=1}^l \psi_{j}(\delta)\cup \psi_{\ast}(\delta)
\end{align}
 for $\delta\in C_u.$ In the case $\{z\in B(0, \eta)\ :\ f_{\gamma}(z)=0\}= \emptyset$ for $\gamma\in G^{\mathbb{N}},$ there exists an initial word $u=u(\gamma)=\gamma_1\cdots \gamma_m\in G^m$ such that for any $\delta\in C_u,$
\begin{align*}
\Psi(\delta)=\psi_{\ast}(\delta)\subset U_{\ast}.
\end{align*} 
Below, for any set-valued map $\psi$ on $C_{u}$ let \[\psi(C_u)=\bigcup_{\delta\in C_u} \psi(\delta)\] for ease of notation. We claim the following.

\begin{lem}
\label{kcomp}
$\psi_j(C_u)$ has at most $k_j$ connected components. 

\end{lem}
\begin{proof}
In this proof, we write $\psi, k$ for $\psi_j, k_j$. Assume that $\psi(C_u)$ splits into $k^{\prime}>k$ disjoint closed and open sets $D_1,..., D_{k^{\prime}}.$ For each set $K\subset \{1,..., k^{\prime}\}$ set \[C^K=\{\delta\in C_u\ :\ \psi(\delta)\cap D_i\neq \emptyset \ \text{for any}\ i\in K\}.\] Then at least two of the $C^K$ are non-empty, and the $C^K$ are disjoint closed and open sets in $C_u$. We now show this statement. We can assume that $k^{\prime}=k+1$ by setting $D_{k+1}=\cup_{i=k+1}^{k^{\prime}} D_i.$ Note that for all $K\subset \{1,..., k+1\},$ the sets $C^K$ are closed and open sets since the set valued function $\psi$ is continuous and the sets $D_i$ are closed and open. Furthermore, for any $i=1,..., k+1$ the set $C^{\{i\}}$ is non-empty by the assumption. Take $C^{\{1\}}$ and consider the set $C^{\{2,3,..., k+1\}}.$ If $C^{\{2,3,..., k+1\}}$ is non-empty, $C^{\{1\}}$ and $C^{\{2,3,..., k+1\}}$ satisfy the desired properties since $C^{\{1\}}\cap C^{\{2,3,..., k+1\}}=\emptyset$. In the case $C^{\{2,3,..., k+1\}}=\emptyset,$ we have $C^{\{2\}}\cap C^{\{3,..., k+1\}}=\emptyset.$ If $ C^{\{3,..., k+1\}}\neq\emptyset,$ $C^{\{2\}}$ and $C^{\{3,..., k+1\}}$ satisfy the desired properties. If $ C^{\{3,..., k+1\}}=\emptyset,$ we consider the sets $C^{\{3\}}$ and $C^{\{4,..., k+1\}}.$ This procedure eventually provides desired sets.

We consider such a $C^K$ which is the union of finitely many sub-cylinders $C_{u^1},..., C_{u^l}$ of $C_u.$ We assume that the words $u^1,..., u^l$ have a common minimal length $m$. Then by construction, $C_{u^j}\subset C^K$ for all $j=1,..., l$, where $u^j=u_1^j\cdots u_m^{j}$. We now prove for all $j=1,..., l$ and all $a\in G,$ 
\begin{align}
\label{hougan2}
C_{u_1^j\cdots u_{m-1}^{j}a}\subset C^K.
\end{align}
 Indeed, for $u^j_m$ and $a$ there exist $c_1,..., c_k$ satisfying the Definition \ref{graph1}. Then it follows \[ C_{u_1^j\cdots u_{m-1}^{j}c_2}\cap C^K\neq \emptyset\] from the fact $C_{u_1^j\cdots u_{m-1}^{j}c_1}\subset C^K$ and Lemma \ref{keylem}. Since $C^K$ is the union of cylinders of length $m$, we have \[C_{u_1^j\cdots u_{m-1}^{j}c_2}\subset C^K.\] If we repeat this procedure, we obtain (\ref{hougan2}), which implies for all $j=1,..., l,$
\[C_{u_1^j\cdots u_{m-1}^{j}}\subset C^K.\]
But this contradicts the minimality of $m$.
\end{proof}

Then the following holds.
\begin{pro}
$X^G$ is locally connected.
\end{pro}
\begin{proof}
Since $\{C_{u(\gamma)}\ :\ \gamma\in G^{\mathbb{N}}\}$ is an open covering of the compact space $G^{\mathbb{N}},$ there is a finite subcovering $C_{u(\gamma_1)},..., C_{u(\gamma_m)}$ such that \[X^G= \bigcup_{j=1}^m \Psi(C_{u(\gamma_j)}).\] By (\ref{spl1}), the set $\Psi(C_{u(\gamma_j)})$ satisfies \[\Psi(C_{u(\gamma_j)})=\bigcup_{i=1}^{l_j}\psi_{j, i}(C_{u(\gamma_j)})\cup \psi_{\ast}(C_{u(\gamma_j)})\] and $\bigcup_{i=1}^{l_j}\psi_{j, i}(C_{u(\gamma_j)})$ has at most $M$ compact connected components of diameter less than $\epsilon$ by above arguments. Since \[ \psi_{\ast}(C_{u(\gamma_j)})\subset \overline{U_{\ast}}:=\{z\in \mathbb{D}\ :\ \eta-\epsilon\le |z|\}\subset X^G,\] we have \[X^G=\bigcup_{j=1}^m \bigcup_{i=1}^{l_j}\psi_{j, i}(C_{u(\gamma_j)})\cup \overline{U_{\ast}}.\]Moreover, $\overline{U_{\ast}}\cap \{z\in \mathbb{D}\ :\ |z|\le L\}=\emptyset$ implies that

\[X^G\cap \{z\in \mathbb{D}\ :\ |z|\le L\}\subset \bigcup_{j=1}^m \bigcup_{i=1}^{l_j}\psi_{j, i}(C_{u(\gamma_j)}).\] Hence each point $z\in X^G\cap \{z\in \mathbb{C}\ :\ |z|\le L\}$ has a connected neighborhood of diameter less than $2\epsilon.$ Combining this and the fact $\{z\in \mathbb{D}\ :\ L<|z|\}\subset X^G$ we have $X^G$ is locally connected.

\end{proof}

\section{Proof of Main Theorem A}
In this section we give a proof of Main Theorem A. Set $I_n:=\{0,1,...,n-1\}$ for $n\ge 2.$ Define a set $\on$ of coefficients which corresponds to $\mn$ as 
\begin{align*}
&\Omega_n:=\left\{\frac{{\xi_n}^j-{\xi_n}^k}{1-{\xi_n}}\ :\ j,k\in I_n\right\}=\frac{\triangle H_n}{1-\xi_n},
\end{align*}
where
\[H_n:=\{{\xi_n}^j\ :\ j\in I_n\}.\]
Then the following two lemmas can be found in \cite{BanHu}.
\begin{lem}\em{\cite[Remark 3]{BanHu}}
\label{mnpo}For any $n\ge 2,$
\begin{align*}
\mn=X^{\on}.
\end{align*}
\end{lem}
\begin{lem}\em{\cite[Proposition 3]{BanHu}}
\label{rcontain}
For any $n\ge 2,$
\begin{align*}
\left\{z\in \mathbb{C}\ :\ \frac{1}{\sqrt{n}}< |z|<1\right\}\subset \mn. 
\end{align*}
\end{lem}
We give a proof of Main Theorem A.
\begin{proof}[Proof of Main Theorem A]
By Main Theorem B, Lemmas \ref{mnpo} and \ref{rcontain}, it suffices to prove that the graph $(\on, R_{\on})$ is connected. 

We first prove that ${\xi_n}^i\in R_{\on}$ for any $i\in I_n.$ Since ${\xi_n}^i H_n=H_n,$ we have that ${\xi_n}^i \on=\on$. Moreover, $\on\subset \triangle \on$ since $\on$ contains $0.$ Hence we have that ${\xi_n}^i\in R_{\on}$. 

Let $j, k\in I_n.$ We have 
\begin{align*}
\frac{{\xi_n}^j-{\xi_n}^{k}}{1-{\xi_n}}-\frac{{\xi_n}^{j+1}-{\xi_n}^{k}}{1-{\xi_n}}=\frac{{\xi_n}^j-{\xi_n}^{j+1}}{1-{\xi_n}}={\xi_n}^j\in R_{\on},\notag \\
\frac{{\xi_n}^j-{\xi_n}^{k+1}}{1-{\xi_n}}-\frac{{\xi_n}^{j}-{\xi_n}^{k}}{1-{\xi_n}}=\frac{{\xi_n}^k-{\xi_n}^{k+1}}{1-{\xi_n}}={\xi_n}^k\in R_{\on}.
\end{align*}
Hence $(\on, R_{\on})$ has an edge between ${({\xi_n}^j-{\xi_n}^{k})}/{(1-{\xi_n})}$ and ${({\xi_n}^{j+1}-{\xi_n}^{k})}/{(1-{\xi_n})}$, and also an edge between ${({\xi_n}^j-{\xi_n}^{k})}/{(1-{\xi_n})}$ and ${({\xi_n}^{j}-{\xi_n}^{k+1})}/{(1-{\xi_n})}.$ This implies that the graph $(\on, R_{\on})$ is connected.
Hence we have proved our theorem.
\end{proof}


\begin{thebibliography}{99}
      \bibitem{Ban1}C. Bandt, {\em On the Mandelbrot set for pairs of linear maps}, Nonlinearity 15(2002), 1127-47.
		\bibitem{BanHu} C. Bandt and N. V. Hung, {\em Fractal n-gons and their Mandelbrot sets}, Nonlinearity 21(2008), 2653-2670.
		\bibitem{BH} M.F.Barnsley and A.N. Harrington, {\em A Mandelbrot set for pairs of linear maps}, Physica 150 (1985) 421-432.
		
		\bibitem{Bou}T. Bousch, {\em Paires de similitudes $z\rightarrow sz+1, z\rightarrow sz-1$}, Preprint (1988).
	
		\bibitem{Bou3}T. Bousch, {\em Sur quelques probl\`{e}mes de dynamique holomorphe}, th\`{e}se de l'Universit\'{e} Paris-Sud, Orsay (1992).
		\bibitem{Bou2}T. Bousch, {\em Connexit\'{e} locale et par chemins h\H{o}lderiens pour les syst\'{e}mes it\'{e}r\'{e}s de fonctions},  Preprint (1993).
		\bibitem{CKW} D. Calegari, S. Koch and A. Walker, {\em Roots, Schottky semigroups, and a proof of Bandt's conjecture}, Ergod. Th \& Dynam. Sys.37, no.8 (2017), 2487-2555.
\bibitem{Fal}K. Falconer, {\em Fractal geometry-Mathematical foundations and applications (Third edition)}, WILEY, 2014.
	
	  \bibitem{HI} Y. Himeki and Y. Ishii, {\em $\mathcal{M}_4$ is regular closed}, Ergod. Th \& Dynam. Sys. 40, no.1 (2020), 213-220.
\bibitem{Hut} J. Hutchinson, {\em Fractals and Self-Similarity}, Indiana Univ. Math. J. 30, no. 5 (1981), 713-747.
 \bibitem{OP} A. M. Odlyzko and B. Poonen, {\em Zeros of polynomials with $0, 1$ coefficients}, L'Enseignement Math.39 (1993), 317-348.
\bibitem{SS}P, Shmerkin and B, Solomyak, {\em Zeros of $\{-1,0,1\}$ power series and connectedness loci for self-affine sets}, Exp. Math. 15 (2006), 499-511.
\bibitem{ST} V. F. Sirvent and J. M. Thuswaldner, {\em On the Fibonacci-Mandelbrot set}, Indagationes Mathematicae 26.1 (2015), 174-190.
\bibitem{SX} B. Solomyak and H. Xu, {\em On the `Mandelbrot set' for a pair of linear maps and complex Bernoulli convolutions}, Nonlinearity 16 (2003), 1733-1749.





\end{thebibliography}
\end{document}